\documentclass[11pt]{amsart}
\usepackage{amsmath, amsthm}
\usepackage[english]{babel}
\usepackage[T1]{fontenc}
\usepackage[latin1]{inputenc}
\usepackage{amssymb}
\usepackage{amscd}
\usepackage{latexsym}
\usepackage[bookmarksnumbered,plainpages,hypertex]{hyperref}
\usepackage{graphicx}
\usepackage{mathrsfs} 

\newtheoremstyle{theorem}
  {12pt}          
  {12pt}  
  {\sl}  
  {\parindent}     
  {\bf}  
  {. }    
  { }    
  {}     
\theoremstyle{theorem}
\newtheorem{theorem}{Theorem}
\newtheorem{corollary}[theorem]{Corollary}
\newtheorem{remark}[theorem]{Remark}
\newtheorem{proposition}[theorem]{Proposition}
\newtheorem{lemma}[theorem]{Lemma}

\newtheorem{definition}[theorem]{Definition}

\linethickness{0.5pt}

\newcommand{\ic}{\ensuremath{\mathcal{I}}}

\newcommand{\oc}{\ensuremath{\mathcal{O}}}

\newcommand{\cc}{\ensuremath{\mathcal{C}}}

\newcommand{\Sc}{\ensuremath{\mathcal{S}}}
\newcommand{\tc}{\ensuremath{\mathcal{T}}}

\newcommand{\Pt}{\mathbb{P}^3}
\newcommand{\PtD}{\mathbb{P}_3^*}

\newcommand{\Pu}{\mathbb{P}^1}

\newcommand{\bI}{\mathbb{I}}

\newcommand{\bQ}{\mathbb{Q}}

\newcommand{\aG}{\alpha}

\newcommand{\cG}{\gamma}

\newcommand{\lG}{\lambda}

\newcommand{\fG}{\varphi}

\newcommand{\oL}{\overline}

\newcommand{\bds}{\begin{displaystyle}}
\newcommand{\eds}{\end{displaystyle}}

\title[Complete intersections primitive structures.]{Complete intersections primitive structures on space curves.}

\author{Ph. Ellia}
\address{Dipartimento di Matematica, 35 via Machiavelli, 44100 Ferrara}
\email{phe@unife.it}


\date{September 12, 2014}

\begin{document}
\maketitle

\begin{abstract} A multiple structure $X$ on a smooth curve $C \subset \Pt$ is said to be \emph{primitive} if $X$ is locally contained in a smooth surface. We give some numerical conditions for a curve $C$ to be a primitive set theoretical complete intersection (i.e. to have a primitive structure which is a complete intersection).
\end{abstract}

\thispagestyle{empty}


\section*{Introduction.}

It is a long standing problem to know whether or not every smooth, irreducible curve $C \subset \Pt$ ($\Pt = \Pt _k$, $k$ algebraically closed, of characteristic zero) is a set theoretic complete intersection (s.t.c.i.) of two surfaces $F_a, F_b$. We recall that $C$ is a s.t.c.i. of $F_a, F_b$ if $F_a \cap F_b = C$ as sets, i.e. $\sqrt{(F_a, F_b)} = \bI (C)$, where $\bI (C)=H^0_*(\ic _C)$. It turns out that this is equivalent to the existence of a \emph{multiple structure} $X$ on $C$ which is the complete intersection of $F_a$ and $F_b$. A multiple structure on $C$ is a locally Cohen-Macaulay curve whose support is $C$ and such that $\deg (X) = md$, where $d = \deg (C)$. The integer $m$ is called the \emph{multiplicity} of $X$.

Following \cite{Banica-Foster} we can distinguish three cases:
\par
a) $X$ is a \textit{primitive} structure: $X$ is locally contained in a smooth surface; in our case this means that at each point $x\in C$ one of the two surfaces $F_a,F_b$ is smooth. 
\par
b) $X$ is a \textit{quasi-primitive} structure: $X$ is generically contained in a smooth surface; in our case this means that there exists a finite subset $T \subset C$ such that for $x\in C\setminus T$, one of the two surfaces $F_a,F_b$ is smooth at $x$; if $y\in T$, both surfaces are singular at $y$.
\par
c) $X$ is a \textit{thick} structure: $X$ contains the first infinitesimal neighbourhood of $C$; in our case this means that both surfaces $F_a,F_b$ are singular along $C$.

\begin{definition}
\label{D-prim s.t.c.i.}
We say that $C$ is a primitive s.t.c.i. of type $(a,b)$, multiplicity $m$, if there exists a primitive structure of multiplicity $m$ on $C$ which is the complete intersection of two surfaces of degrees $a, b$.
\end{definition}

In this note we will consider only primitive s.t.c.i. First we show (Proposition \ref{P-finite abm}) that given any curve $C$ of degree $d$, genus $g$, there are only finitely many possible $(a,b,m)$ for $C$ to be a primitive s.t.c.i. Now assume $X$ is a primitive multiple structure on $C$ which is a complete intersection: $X = F_a\cap F_b$, $a \leq b$. If $a < b$, $F_a$ is uniquely defined, but $F_b$ moves. For instance $X$ is cut-off schematically by the surfaces of degree $b$. In particular the general $G \in H^0(\ic _X(b))$ is smooth outside of $C$. We want to use this freedom to control the singularities of $G$ along $C$. The main result of this note (Theorem \ref{T-uniformity}) states that there exist integers $n,k$ such that for every $p \in Sing(G)$, $(G,C)_p$ is a singularity of type $A^k_n$ (see Definition.\ref{D-Akn}). With this result we are able to give (Theorem \ref{T-final}) some new numerical conditions for $C$ to be a primitive s.t.c.i. When applied to the case where $C$ is a smooth rational quartic curve, these conditions yield (Corollary \ref{C-list 2 quartic}) a slight improvement on a earlier result of Jaffe (\cite{Jaffe}). Finally, and unfortunately, we show (Remark \ref{Rmk-end}) that the results of this note are not sufficient for giving a single example of a curve which is not a primitive s.t.c.i. Something else is needed.

\noindent As the reader will see, this note owes a lot to Jaffe's work (\cite{Jaffe}, \cite{Jaffe-local}), so thank you to him. I also thank Massimiliano Mella for pointing to me the paper \cite{Kollar}.

\section{Primitive structures.}

Let $C \subset \Pt$ be a smooth, irreducible curve of degree $d$, genus $g$. Assume there is a primitive structure $X$, of multiplicity $m$, on $C$ which is the complete intersection of two surfaces $F_a, F_b$. Let $a=\deg (F_a), b=\deg (F_b)$, $a \leq b$.

The Cohen-Macaulay filtration $C=X_0\subset X_1\subset ... \subset X_{k}=X$ is defined by $X_i =C^{(i)}\cap X$ where $C^{(i)}$ is the i-th infinitesimal neighbourhood of $C$ ($\ic _{C^{(i)}} = \ic _C^{i+1}$).
\par \noindent
If $X$ is \textit{primitive}, $X_i$ is locally CM, in fact even locally complete intersection of multiplicity $i+1$ (locally given by $(x,y^{i+1})$ in suitable coordinates). In particular $X_1$ is a double structure corresponding to a quotient, $L^*$, of the conormal bundle $N^*_C$, and we will say that $X$ is a primitive structure of type $L$. For $1\leq i \leq m-1$, we have exact sequences:
$$0 \to L^{*\otimes i} \to \oc _{X_i} \to \oc _{X_{i-1}} \to 0$$
It follows that:
\begin{equation}
\label{eq:pa_prim}
p_a(X)= g -\sum_{i=1}^{m-1} \chi (L^{*\otimes i})     
\end{equation}
hence:
\begin{equation}
\label{eq:pa-prim 2}
p_a(X) = m(g-1)+1+lm(m-1)/2
\end{equation}
where $l = \deg (L)$, $L \subset N_C$.

Of course, since $X = F_a\cap F_b$, we also have:
\begin{equation}
\label{eq:pa_ab}
p_a(X) = 1+{ab(a+b-4)\over 2}  
\end{equation}
and:
\begin{equation}
\label{eq:4mab}
dm = ab  
\end{equation}

Combining everything we get:

\begin{equation}
\label{eq:basic rel}
2g-2 + l(m-1) = d(a+b-4)
\end{equation}

We recall the following:

\begin{lemma}
\label{L-gauss map}
Let $C \subset \Pt$ be a smooth, irreducible curve of degree $d$, genus $g$. Assume $C$ is not contained in a plane. If $L \subset N_C$, then:
\begin{equation}
\label{eq:bound l}
l = \deg (L) \leq 3d + 2g - 4
\end{equation}
\end{lemma}

\begin{proof} The line bundle $L(-1) \subset N_C(-1)$ comes from a rank two sub-bundle $E_L \subset T_{\Pt}(-1)|C$ containing $T_C(-1)$. The bundle $E_L$ gives a Gauss map $\fG _L:C \to \PtD$. If $D = \fG _L(C)$, then $D$ is a curve of degree $\geq 2$. We have $\deg (\fG _L).\deg (D)= 3d+2g-2-l$ (see \cite{Hulek-Sacchiero}). The conclusion follows.
\end{proof}

\begin{proposition}
\label{P-finite abm}
Let $C \subset \Pt$ be a smooth, irreducible curve of degree $d$, genus $g$. There exist finitely many $(l,m,a,b)$ satisfying (\ref{eq:basic rel}). In other words there exist finitely many possible $(a,b,m)$ for $C$ to be a primitive s.t.c.i. of surfaces of degrees $a,b$, with multiplicity $m$.
\end{proposition}

\begin{proof} If $l \leq 0$, from (\ref{eq:basic rel}) we get $a+b-4 \leq (2g-2)/d$. Hence we have finitely many $(a,b)$. Since $m = ab/d$, each $(a,b)$ determines $m$ and $l$. From now on let us assume $l>0$ and $a \leq b$. By Lemma \ref{L-gauss map}, we have finitely many possible values of $l$. We may rewrite (\ref{eq:basic rel}) as follows:
\begin{equation}
\label{eq:basi rel rew}
a = \aG (1+ \frac{a}{b}-\frac{4}{b})+\frac{\cG}{b}
\end{equation}
where $\aG = d^2/l$, $\cG = 2d(1-g)/l + d$. Clearly $\cG \leq 3d$. Hence $\cG /b \leq d\frac{3}{b} \leq d$ (we may assume $b \geq 3$). Now since $a \leq b$, $1+a/b - 4/b < 2$. It follows that $a < 2\aG +d =2d^2/l + d$. So for each value of $l$ we have finitely many possible $a$. We conclude since to any fixed $(a,l)$ there corresponds a unique $b$. 
\end{proof}

The first candidate for a non s.t.c.i. curve in $\Pt$ is a rational quartic curve. In this case we have:

\begin{lemma}
\label{L-1 list quartic}
let $C \subset \Pt$ be a smooth rational quartic curve. If $C$ is a primitive s.t.c.i. of type $(a,b,m,l)$ then $(a,b,m,l)$ is one of the following fifteen cases:

\begin{center}
\begin{tabular}{|c||c|c|c|c|c|c|c|c|} \hline
$l$ & 7 & 6 & 6 & 5 & 3 & 2 & 2 & 2 \\ \hline
$(a,b)$ & (3,4) & (3,8) & (4,4) & (4,7) & (6,26) & (13,16) & (12,18) &(10,28)\\ \hline
$m$ & 3 & 6 & 4 & 7 & 39 & 52 & 54 & 70 \\ \hline
\end{tabular}
\medskip

\begin{tabular}{|c||c|c|c|c|c|c|c|} \hline
$l$ & 2 & 1 & 1 & 1 & 1 & 1 & 1 \\ \hline
$(a,b)$ & (9,48) & (28,33) & (22,50) &(20,67)&(19,84) & (18,118)& (17,220) \\ \hline
$m$ & 108 & 231 & 275 & 335 & 399 & 531 & 935 \\ \hline
\end{tabular}
\end{center}
\end{lemma}

\begin{proof} A sub-bundle of $N_C \simeq 2\oc _{\Pu}(7)$ has degree at most $7$. From (\ref{eq:basic rel} we get $l>0$, so $1\leq l \leq 7$. Equation (\ref{eq:basic rel}) reads like: $lm = 4(a+b)-14+l$. Since $m = ab/4$, we get: $b(la-16) = 16a -56+4l$, hence $la = 16 + 16a/b -(56 -4l)/b$. Since $a \leq b$, we get: $a < 32/l$. Now for each $l$, $1 \leq l \leq 7$ and for each $a$, $3 \leq a < 32/l$, we compute $b = (16a -56 +4l)/(la-16)$ (n.b. the case $la-16=0$ is impossible). This can be done by hand but it is faster with a computer.
\end{proof}

\begin{remark} Proposition \ref{P-finite abm} and Lemma \ref{L-1 list quartic} were first proved by Jaffe in \cite{Jaffe} by a different method.
\end{remark}

\section{Singularities and normal bundle.}

Let $C\subset S \subset \Pt$ be a smooth curve on the surface $S$ of degree $s$. Assume $dim(C \cap Sing(S))=0$. The inclusion $C \subset S$ yields $\oc (-s) \to \ic _C$, which, by restriction to $C$, gives $\sigma :\oc _C \to N^*_C(s)$. The section $\sigma$ vanishes on $C \cap Sing(S)$. More precisely we have an exact sequence:
$$0 \to \oc _C(-s) \to N^*_C \to L^* \oplus T \to 0$$
where $L$ is locally free of rank one and where $T$ is a torsion sheaf with support on $C \cap Sing(S)$. This exact sequence defines a surjection $N^*_C \to L^* \to 0$ which in turn defines a double structure on $C$; this double structure is nothing else than $C$ "doubled on $S$" (i.e. the greatest locally Cohen-Macaulay subscheme of $C^{(1)} \cap S$ or, equivalently the subscheme of $S$ defined by $\ic ^{(2)}$ (symbolic power) where $\ic$ is the ideal of $C$ in $S$).
\par
From the exact sequence above we get: $deg(T) = deg(L) - (2g-2+d(4-s))$. Notice that: $deg(T) = deg(L)-deg(\omega _C(4-s))$ and that if $S$ were smooth we would have $N_{C,S} \simeq \omega _C(4-s)$; in this sense we can think to $L$ as the "normal bundle" of $C$ in $S$, with respect with the smooth case, $\deg (L)$ gets a contribution from the singularities of $S$ lying on $C$.

\begin{definition}
\label{def_n(S,C)}
In the above situation we define $n(S,C)$, the contribution of the singularities of $S$ to the normal bundle of $C$, by: $n(S,C):=deg(L)+d(s-4)-2g+2$.
\end{definition}

\begin{remark}
In the terminology of \cite{Jaffe}, $n(S,C)=p_1(S,C)$. Of course, the computation of $n(S,C)$ is a local problem: $n(S,C)=\sum_{p \in Sing(S)\cap C} n(S,C)_p$.
\end{remark}

\begin{definition}
\label{D-Akn}
Following \cite{Jaffe}, a surface-curve pair $(S,C)$ is a surface $S$ with a curve $C\subset S$ such that $C$ is a regular scheme not contained in $Sing(S)$. Given a surface-curve pair, one may consider, for every $p$ in $C$ the \emph{local} pair $(S,C)_p$; this amounts to give the data $(A,I)$ where $A = \oc _{S,p}$ and where $I \subset A$ is the ideal of $C$. Assume $p$ is a singularity of type $A_n$. Let $E_1,...,E_n$ (with the natural numbering) denote the exceptional curves over $p$ in the minimal resolution $f:S' \to S$. Then $\overline{C}$, the strict transform of $C$, meets a unique exceptional curve $E_k$ and $\overline{C}.E_k=1$. In this case we will say that $(S,C)_p$ is a singularity of type $A_n^k$.
\end{definition}
If $p$ is an $A_n$ singularity, every $(S,C)_p$ is analytically isomorphic to some $A_n^k$ for some $k \leq \frac{n+1}{2}$ (\cite{Jaffe-local}).

\begin{lemma}
\label{CExAnk}
Assume $(S,C)_p$ is an $A_n^k$ singularity, then $\overline{C}.E=\dfrac{k(n+1-k)}{(n+1)}$.
\end{lemma}

\begin{proof}
We have $f^*(C)=\overline{C}+\sum_{i=1}^n f_iE_i$ ($E=\sum_{i=1}^n f_iE_i$), writing $f^*(C).E_j=0=\delta _{jk}+\sum_{i=1}^n f_iE_iE_j$, for $1\leq j\leq n$, we get a linear system in the $f_i$'s; solving this system, we get $f_k=\frac{k(n+1-k)}{(n+1)}$.
See \cite{Jaffe} Prop.5.8 for more details.
\end{proof}

\begin{lemma}
\label{p1_Ank}
If $(S,C)_p$ is an $A_n^k$ singularity with $k \leq {n+1\over 2}$, then $n(S,C)_p=k$.
\end{lemma}

\begin{proof}
Since the question is local we may assume $S$ given by $xy+yz^k+xz^{n-k+1}=0$ and $C$ given by $I=(x,y)$ (cf \cite{Jaffe}). Looking at the equation of $S$ mod.$I^2$, we get $n(S,C)_p=min\{k, n-k+1\}=k$ under our assumption.
\end{proof}

We will need also the following

\begin{lemma}
\label{CE}
Let $C \subset F_b$. Assume that $F_b$ is normal with only \emph{rational double points} ($A$, $D$, $E$ double points). Let $f:\overline{F} \to F$ be the minimal resolution and let $f^*C = \overline{C}+E$ where $\overline{C}$ is the strict transform of $C$ and where $E$ is an effective $\bQ$ divisor supported on the exceptional locus.
With notations as above, if $C$ is a set theoretic complete intersection on $F_b$, then $\overline{C}.E= d(b-4)+\dfrac{d^2}{b}+2-2g$.
\end{lemma}

\begin{proof}
Since $F_b$ has only rational singularities, $\omega _{\oL F} = f^*(\omega _{F_b}) \simeq f^*(\oc _{F_b}(b-4))$. It follows that $\oL C.\oL K = d(b-4)$. By adjunction: $\oL C^2 + \oL C\oL K= 2g-2$ and we get $\oL C^2= 2g-2-d(b-4)$. If $C$ is a set theoretic complete intersection on $F_b$, then $f^*(mC) = f^*(\oc _{F_b}(a))$. It follows that $f^*(C) = \dfrac{a}{(b-4)m}\oL K$, as $\bQ$-divisor. Now, on the one hand: $f^*(C).\oL C = (\oL C+E)\oL C = 2g-2-d(b-4)+\oL C.E$ and on the other hand $f^*(C).\oL C = \dfrac{a}{(b-4)m}\oL C.\oL K = \dfrac{d^2}{b}$ (using $md=ab$). Combining the two we get the result.
\end{proof}

\section{A uniformity result.}

Assume $X$ is a primitive multiple structure on $C$ which is a complete intersection: $X = F_a\cap F_b$, $a \leq b$. If $a < b$, $F_a$ is uniquely defined, but $F_b$ moves: we may replace it by $F_aP+\lG F_b$. For instance $X$ is cut-off schematically by the surfaces of degree $b$. In particular the general $G \in H^0(\ic _X(b))$ is smooth outside of $C$. We want to use this freedom to control the singularities of $G$ along $C$, where $G \in H^0(\ic _X(b))$ is general. We have:

\begin{theorem}
\label{T-uniformity}
Let $C\subset \Pt$ be a smooth, irreducible curve. Assume there exists a primitive structure, $X$, on $C$ such that $X=F_a\cap F_b$, with $a\leq b$. Then the general surface, $G$, of degree $b$ containing $X$ is normal with at most rational singularities along $C$. More precisely, $G$ is smooth outside of $C$ and there exist integers $n,k$ such that for every $p\in Sing(G)$, $(G,C)_p$ is a singularity of type $A_n^k$.
\end{theorem}

\begin{proof}
Since $\ic _X(b)$ is generated by global sections, by Bertini's theorem, the general surface of degree $b$ containing $X$ is smooth outside of $C$, so we may assume that $F_b$ is smooth out of $C$. Consider a general pencil, $\Delta$, of degree $b$ surfaces containing $X$: $\lambda F_a.P+\mu F_b$. As already observed, the general surface of this pencil is smooth outside of $C$. Moreover, since $X$ is primitive, for every $x\in C$ there exists a surface of the pencil which is smooth at $x$, by \cite{Kollar} this implies that the general surface, $S$, of the pencil has at most singularities of type $A_n$. Indeed by \cite{Kollar} Theorem 4.4 $S$ has only $cA$ singularities. But normal $cA$ singularities are canonical and since we are dealing with normal singularities of surfaces, they are rational (see Theorem 3.6 of \cite{Kollar}). Finally since the quadratic part of a local equation has rank at least 2, they are of type $A_n$.
\par
The pencil $\Delta \simeq \Pu$ gives a family $\mathcal{F} \subset \Pt _{\Delta} \to \Delta$ of degree $b$ surfaces with only rational singularities. We may consider (after a base-change if necessary) a simultaneous resolution of the singularities, this resolution is obtained by successive blow-ups (\cite{777}). After a certain number of blow-up, the general fiber of $\mathcal{F} \to \Delta$ is smooth; this shows that there exists $n$ such that the general member of $\Delta$ has only $A_n$ singularities.
\par
We have a morphism $g: C \to \Delta:x \to F_x$ where $F_x$ is the unique surface of $\Delta$ which is singular at $x$. Let $Y \subset C \times \Delta$ be the graph of $g$. Then $Y$ is a unisecant on the ruled surface $p: C \times \Pu \to C$, (we have identified $\Delta$ to $\Pu$) hence $Y$ is smooth, irreducible. In $\Pt \times \Pu$ consider the incidence $\Sc =\{(x,t) \mid x \in F_t\}$. The inclusion $\cc =C \times \Pu \subset \Sc$ gives a morphism $\ic _{\Sc} \to \ic _{\cc }$. Restricting to $\cc $, we get: $\varphi: p^*(\oc (-b))\otimes q^*(\oc (-1)) \to p*(N^*_C)$, when restricted to a fiber of $q: C \times \Pu \to \Pu$, this morphism is the morphism $\oc (-b) \to N^*_C$ induced by $F_t$. The cokernel of $\varphi$ has a torsion subsheaf, $\tc$, supported on $Y$. Since $\tc$ is locally free on an open subset of $Y$, we conclude (using Lemma \ref{p1_Ank}) that there exists $k$ such that the general member of $\Delta$ has at most $A_n^k$-singularities along $C$.
\end{proof}

A first consequence:

\begin{lemma}
\label{alpha_k}
Let $X$ be a primitive structure of type $L$ on a smooth, irreducible curve $C$, of degree $d$, genus $g$. Assume $X = F_a \cap F_b$, $a \leq b$. Assume $F_b$ is general in the sense of Theorem \ref{T-uniformity}, i.e. there exist $n$, $k$ such that $F_b$ has $\alpha$ singularities of type $A_n^k$ on $C$ and is smooth outside of $C$. Then $\alpha k=d(b-4) -2g + 2 +l$.
\end{lemma}

\begin{proof}
We have $n(F_b,C)= \deg (L) -\deg \omega _C(4-b)= l-2g+2+d(b-4)$ and the conclusion follows from Lemma \ref{p1_Ank}.
\end{proof}

We have other numerical conditions:

\begin{proposition}
\label{num_cond}
Assume $X$ is a multiplicity $m$ primitive structure of type $L$ on a smooth, irreducible curve $C$ such that $X = F_a \cap F_b$ with $a \leq b$, $b > 4$. Then there exist positive integers $\alpha , n, k$ with $k \leq {n+1\over 2}$ satisfying the following conditions:
\begin{enumerate}
\item $\alpha k= d(b-4)-2g+2+l$
\item $n+1= \dfrac{\aG bk^2}{lb-d^2}$
\item $\alpha n < \dfrac{\alpha n(n+2)}{n+1} \leq \dfrac{2b(b-1)^2}{3}$
\item In particular: $\dfrac{(A+l)(bA+d^2)}{lb-d^2} < \dfrac{2b(b-1)^2}{3}$, where: $A := d(b-4)-2g+2$.
\end{enumerate}
\end{proposition}

\begin{proof}
We may assume that $F_b$ is general in the sense of Theorem \ref{T-uniformity}, hence there exist $n$, $k$ such that $F_b$ is normal and has $\alpha$ singularities of type $A_n^k$ on $C$; this defines the numbers $\alpha, n, k$.
\par
Condition (1) is Lemma \ref{alpha_k}.
\par
By Lemma \ref{CE}, $\overline{C}.E= A+(d^2/b)$ ($A := d(b-4)-2g+2$). Under our assumption and using Lemma \ref{CExAnk} we get $\overline{C}.E={\alpha k(n+1-k)\over n+1}$. It follows that $A+(d^2/b) = {\alpha k(n+1-k)\over n+1}$ (*). The difference (1)-(*) yields $(\aG k^2)/(n+1) = (bl-d^2)/b$ and (2) follows.
\par
From Miyaoka's inequality (\cite{Miyaoka} Cor.1.3 with $D=0$) we get $\sum_{p\in Sing(F_b)} \nu (p) \leq c_2(\tilde{F}_b)-{c_1^2(\tilde{F}_b)\over 3}$. For an $A_n$ singularity, $\nu =(n+1 - {1\over n+1})$ (the Euler number is $n+1$ and $\mid G\mid = n+1$). Since $c_1^2(\tilde{F}_b)=b(b-4)^2$ and since $c_2(\tilde{F}_b)=b^3-4b^2+6b$, we get ${\alpha n(n+2)\over n+1} \leq {2b(b-1)^2\over 3}$ and (3) follows.
\par
From (2) $n+1 = (\aG bk^2)/(lb-d^2)$. From (1) $\aG k=A+l$. It follows that $n = \frac{bk(A+l)}{lb-d^2}-1$. Hence $\aG n = \frac{\aG bk(A+l)}{lb-d^2}-\aG$. From (1), $\aG \leq A+l$. This implies $\aG n \geq \frac{\aG bk(A+l)}{lb-d^2}-(A+l) = (A+l)[b(\aG k-l)+d^2]/(lb-d^2) = (A+l)[bA+d^2]/(lb-d^2)$. We conclude with (3).
\end{proof}

Gathering everything together:

\begin{theorem}
\label{T-final}
Let $C \subset \Pt$ be a smooth, irreducible curve of degree $d$, genus $g$. If $C$ is a primitive set theoretic complete intersection, then there exist integers $a, b, l, m, \aG , n, k$ such that:
\begin{enumerate}
\item $md = ab$
\item $2g-2 +l(m-1) = d(a+b-4)$
\item $\aG k = A + l$, where $A := d(b-4)-2g+2$
\item $\aG n < 2b(b-1)^2/3$
\item $3(A+l)(bA +d^2) < 2b(lb - d^2)(b-1)^2$
\end{enumerate}
\end{theorem}

Applying this to the case $d=4, g=0$ we get:

\begin{corollary}
\label{C-list 2 quartic}
Let $C \subset \Pt$ be a smooth rational quartic curve. If $C$ is a primitive s.t.c.i. of type $(a,b,m,l)$, then $(a,b,m,l)$ is one of the following ten cases:
$(a,b,l) \in \{(4,7,5), (6,26,3), (12,18,2), (10, 28,2), (9,48,2), (22,50,1), (20,67,1),\\(19,84,1), (18,118,1), (17,220,1)\}.$
\end{corollary}

\begin{proof} We take the list of Lemma \ref{L-1 list quartic} and we apply Theorem \ref{T-final}. Thanks to condition (5) we can exclude five cases: $(a,b)= (3,4), (3,8), (4,4), (13, 16), (28, 33)$.
\end{proof}

\begin{remark} The cases with $a=3$ were already excluded in \cite{Jaffe} in a different way, so our improvement is limited to the exclusion of three cases ($(4,4), (13, 16), (28, 33)$).
\end{remark}  

\begin{remark}
\label{Rmk-end} Although for given $(d,g)$ we know that there is a finite number of possible $(a,b,m,l)$ (see Proposition \ref{P-finite abm}) it turns out that for every $(d,g)$, the numerical conditions of Theorem \ref{T-final} are always fulfilled. Indeed let us set $l=1$. From (2) we get: $m = d(a+b-4)-2g+3$. Since $m = ab/d$ by (1), it follows that: $b = (d^2a -4d^2 -2gd + 3d)/(a - d^2)$. Set $a = d^2+1$, then $b = d^4 -3d^2 + 3d -2gd$. Since $g \leq (d-2)^2/4$ by Castelnuovo's bound, we have $0 < a \leq b$. Now take $k=1$, so that $\aG = A+1$ by (3). Condition (4) is fulfilled by $n=1$ (and also by many other values of $n$ since $2b(b-1)^2/(A+1) \sim 2d^7/3$). Finally to check (5) first observe that under our assumptions: $A+1 \leq db$ so it is enough to show that: 
$$0 < 2b^3 -b^2(4+5d^2) +2b(1+4d^2) -2d^2 -3d^3$$
Now since $b \geq d^4-d^3$ we have $2b(1+4d^2) -2dì2-3d^3 > 0$ and $2b^3 -b^2(4+5d^2) > 0$.
\end{remark}




\begin{thebibliography}{Prim}

\bibitem{Banica-Foster} Banica, -Foster, O.: \textit{Multiplicity structures on space curves}, in The Lefschetz centenial conference, part I, Contemp. Math., vol. {\bf 58}, 47-64, Amer. Math. Soc. (1986) 



\bibitem{Hulek-Sacchiero} Hulek, K.-Sacchiero, G.: \textit{On the normal bundle of elliptic space curves}, Arch. Math., {\bf 40}, 61-68 (1983)

\bibitem{Jaffe} Jaffe, D.B.: \textit{Applications of iterated blow-up to set theoretic complete intersections in $\Pt$}, J. reine angew. Math., {\bf 464}, 1-45 (1995)

\bibitem{Jaffe-local} Jaffe, D.B:: \textit{Local geometry of smooth curves passing through rational double points}, Math. Ann., {\bf 294}, 645-660 (1992)

\bibitem{Kollar} Kollar, J.: \textit{Singularities of pairs}, in {\it Algebraic Geometry, Santa Cruz 1995}, Proc. Sympos. Pure Math., {\bf 62}, part 1, 221-287, Amer. Math. Soc. (1997)

\bibitem{Miyaoka} Miyaoka, Y.: "The maximal number of quotient singularities on surfaces with given numerical invariants", Math. Ann.,\textbf{268}, 159-171 (1984)

\bibitem{777} Pinkham, H.: {\it Résolution simultanée de points doubles rationnels} in Demazure, M. et al: \textit{S\'eminaire sur les singularit\'es des surfaces}, L.N.M. {\bf 777}, 179-205, Springer (1980) 



\end{thebibliography}
\end{document}